\documentclass[12pt]{amsart}
\usepackage{amssymb}
\usepackage{hyperref}
\usepackage{tabularx}
\usepackage{booktabs}
\usepackage{caption}

\newtheorem{thm}{Theorem}[section]

\newtheorem{prop}[thm]{Proposition}

\newtheorem{con}[thm]{Conjecture}
\newtheorem*{prob*}{Open problem}

\theoremstyle{definition}

\theoremstyle{remark}

\newtheorem{rem}[thm]{Remark}
\newtheorem*{rem*}{Remark}

\DeclareMathOperator{\id}{id}

\newcommand{\kringel}{\mathbin{\raise1pt\hbox{$\scriptstyle\circ$}}}
\newcommand{\pkt}{\mathbin{\raise0pt\hbox{$\scriptstyle\bullet$}}}

\newcommand{\un}{\underline{n}}

\newcommand{\F}{\mathbb{F}}
\newcommand{\N}{\mathbb{N}}

\newcommand{\ad}{{\rm ad}}

\newcommand{\End}{{\rm End}}
\newcommand{\Der}{{\rm Der}}

\newcommand{\Out}{{\rm Out}}

\newcommand{\La}{\mathfrak{a}}

\newcommand{\Le}{\mathfrak{e}}
\newcommand{\Lf}{\mathfrak{f}}
\newcommand{\Lg}{\mathfrak{g}}
\newcommand{\Lh}{\mathfrak{h}}

\newcommand{\Ls}{\mathfrak{s}}
\newcommand{\Lt}{\mathfrak{t}}

\newcommand{\Lz}{\mathfrak{z}}

\newcommand{\CM}{\mathcal{M}}

\newcommand{\CO}{\mathcal{O}}

\newcommand{\abs}[1]{\lvert#1\rvert}

\newcommand{\al}{\alpha}
\newcommand{\be}{\beta}
\newcommand{\ga}{\gamma}
\newcommand{\de}{\delta}
\newcommand{\ep}{\varepsilon}

\newcommand{\la}{\lambda}
\newcommand{\om}{\omega}

\newcommand{\ov}{\overline}

\newcommand{\ra}{\rightarrow}

\renewcommand{\phi}{\varphi}

\begin{document}


\title[Zassenhaus conjecture]{Counterexamples to the Zassenhaus conjecture on simple modular Lie algebras}

\author[D. Burde]{Dietrich Burde}
\author[W. Moens]{Wolfgang Moens$^{\dagger}$}
\author[P. P\'aez-Guill\'an]{Pilar P\'aez-Guill\'an}
\address{Fakult\"at f\"ur Mathematik\\
Universit\"at Wien\\
  Oskar-Morgenstern-Platz 1\\
  1090 Wien \\
  Austria}
\email{dietrich.burde@univie.ac.at}
\address{Fakult\"at f\"ur Mathematik\\
/Universit\"at Wien\\
  Oskar-Morgenstern-Platz 1\\
  1090 Wien \\
  Austria}
\email{wolfgang.moens@univie.ac.at}
\address{Fakult\"at f\"ur Mathematik\\
Universit\"at Wien\\
  Oskar-Morgenstern-Platz 1\\
  1090 Wien \\
  Austria}
\email{maria.pilar.paez.guillan@univie.ac.at}

\date{\today}

\footnote{The second author died in May $2022$.}

\begin{abstract}
We provide an infinite family of counterexamples to the conjecture of Zassenhaus on the solvability of the outer
derivation algebra of a simple modular Lie algebra. In fact, we show that the simple modular
Lie algebras $H(2;(1,n))^{(2)}$ of dimension $3^{n+1}-2$ in characteristic $p=3$ do not have a solvable outer derivation algebra
for all $n\ge 1$. For $n=1$ this recovers the known counterexample of $\mathfrak{psl}_3(\F)$. We show that the outer
derivation algebra of $H(2;(1,n))^{(2)}$ is isomorphic to $(\mathfrak{sl}_2(\F)\ltimes V(2))\oplus \F^{n-1}$, where $V(2)$ is the
natural representation of $\mathfrak{sl}_2(\F)$. We also study other known simple Lie algebras in characteristic three, but
they do not yield a new counterexample.
 \end{abstract}

\maketitle

\section{Introduction}

We study a conjecture by Hans Zassenhaus, which says that the outer derivation algebra $\Out(\Lg)$ is 
{\em solvable} for all simple modular Lie algebras $\Lg$, over a field $\F$ of characteristic $p>0$.
Zassenhaus posed this conjecture in $1939$ in his work \cite{ZAS}. We have collected several results
on this conjecture from the literature, and proved some results in \cite{BU67}. For simple modular Lie algebras over an
algebraically closed field of characteristic $p>3$ the Zassenhaus conjecture is true. The outer derivation algebra
$\Out(\Lg)$ is solvable of derived length at most three. In characteristic $p=2$ and $p=3$, however, there is a counterexample
known in each case. For $p=3$ this is a  simple constituent of the classical Lie algebra $\Lg_2$, namely $\mathfrak{psl}_3(\F)$.
For $p=2$ it is a simple constituent of dimension $26$ of the classical Lie algebra $\Lf_4$. \\[0.2cm]
One motivation for us to study the Zassenhaus conjecture comes from commutative post-Lie algebra structures,
or {\em CPA-structures}, on finite-dimensional Lie algebras over a field $\F$, see \cite{BU67}. Indeed, every perfect
modular Lie algebra in characteristic $p>2$ having a solvable outer derivation algebra admits only the trivial CPA-structure.
Here CPA-structures are a special case of post-Lie algebra structures on Lie algebras, which have been studied in the context
of geometric structures on Lie groups,
\'etale representations of algebraic groups, deformation theory, homology of partition posets, Kozul operads, Yang-Baxter equations,
and many other topics. For references see \cite{BU41,BU44,BU51,BU52,BU57,VAL}. \\[0.2cm]
In this article we provide an infinite family of new counterexamples to the Zassenhaus conjecture in characteristic $3$.
We show that the Hamiltonian Lie algebras $H(2;(1,n))^{(2)}$, which are central simple modular Lie algebras in characteristic $3$
of dimension $3^{n+1}-2$, are counterexamples for all $n\ge 1$.
For $n=1$ we have the isomorphism $H(2;(1,1))^{(2)}\cong \mathfrak{psl}_3(\F)$, which recovers the known counterexample
in characteristic $3$. We show that there are no other counterexamples among the graded Hamiltonian Lie algebras
$H(2r;\underline{n})^{(2)}$ in characteristic $p\ge 3$. We also determine the structure of the outer derivation
algebra of $H(2r;\underline{n})^{(2)}$ in characteristic $p=3$. Finally, we study the Zassenhaus conjecture for known
simple Lie algebras of {\em non-standard type} over an algebraically closed field of characteristic three, such as Brown's
algebras $Br_8$ and $Br_{29}$, Kostrikin's series $L(\ep,\de,\rho)$ of dimension $10$, the Ermolaev algebras $R(\underline{n})$,
the Brown-Kuznetsov algebras $T(n)$ and several series of new simple Lie algebras of Skryabin. We do not find
new counterexamples there.

\section{Preliminaries}

Let $\Lg$ be a finite-dimensional Lie algebra over an arbitrary field $\F$. Denote by $\Der(\Lg)$ the derivation algebra
of $\Lg$ and by $\ad(\Lg)$ the ideal of inner derivations of the Lie algebra $\Der(\Lg)$. The quotient algebra
$\Out(\Lg)=\Der(\Lg)/\ad(\Lg)$ is called the {\em algebra of outer derivations} of $\Lg$.
Hans Zassenhaus posed in $1939$ in his work \cite{ZAS} on page $80$, between ``Satz $7$'' and ``Satz $8$'',
the following conjecture.

\begin{con}[Zassenhaus]
The outer derivation algebra $\Out(\Lg)$ of a simple Lie algebra $\Lg$ in prime characteristic is solvable.
\end{con}

For the conjecture we sometimes assume that $\Lg$ is defined over an algebraically closed field of
characteristic $p>0$, because then we can apply the classification results.
For characteristic zero, the corresponding conjecture is true, because then $\Out(\Lg)\cong H^1(\Lg,\Lg)=0$ for a simple Lie algebra
$\Lg$ by the first Whitehead Lemma. Clearly, this need not be true in prime characteristic, and indeed the outer derivation algebra
of a simple modular Lie algebra need not be trivial in general. 

\begin{rem}
The Zassenhaus conjecture for Lie algebras can be seen as an analogue of the {\em Schreier conjecture} for finite groups.
The Schreier conjecture asserts that the outer automorphism group of every finite simple non-abelian group is solvable. 
It was proposed by Otto Schreier in $1926$ and is known to be true as a result of the classification of finite 
simple groups. Up to now no simpler proof is known for it.
\end{rem}

What is known about the Zassenhaus conjecture? There are many different results in the literature, in particular,
in the context of the classification of simple modular Lie algebras over an algebraically closed field of characteristic
$p>3$. Let us summarize the main results, which we have collected in \cite{BU67}. A simple modular Lie algebra in the
classification is either of classical type, Cartan type, or of Melikian type in characteristic $p=5$. The results are
as follows.

\begin{prop}
Let $\Lg$ be a classical simple Lie algebra over a field $\F$ of characteristic 
$p>3$. Then $\Out(\Lg)=0$ unless $\Lg=\mathfrak{psl}_{n+1}(\F)$ with $p\mid n+1$ in which case 
$\Der(\Lg)\cong \mathfrak{pgl}_{n+1}(\F)$ and $\Out(\Lg)\cong \F$.
\end{prop}

\begin{prop}
Let $\Lg$ be a simple Lie algebra of Cartan type over a field of characteristic $p>3$.
Then $\Out(\Lg)$ is solvable. More precisely, $\Out(\Lg)$ is solvable of derived length $d\le 1$ for type $W$ and type
$K$, of derived length $d\le 2$ for type $S$ and of derived length $d\le 3$ for type $H$. 
\end{prop}

\begin{prop}
Let $\CM=\CM (n_1,n_2)$ be a Melikian algebra of dimension $5^{n_1+n_2+1}$ over a field of
characteristic $5$. Then $\Out(\CM)$ is abelian.
\end{prop}

So the Zassenhaus conjecture has a positive answer for algebraically closed fields of characteristic $p>3$:

\begin{thm}
Let $\Lg$ be a simple modular Lie algebra over an algebraically closed field of characteristic $p>3$. Then
$\Out(\Lg)$ is solvable of derived length at most three.
\end{thm}

Recall that a Lie algebra $\Lg$ over $K$ is called {\em central simple} if its centroid
coincides with $K$. Here the centroid is the space of all $K$-linear maps $\phi\colon \Lg \ra \Lg$ 
commuting with all inner derivations. If $\Lg$ is a central simple Lie algebra over an arbitrary 
field $\F$ of characteristic $p>3$, then $\Lg \otimes_{\F}\ov{\F}$ is simple over $\ov{\F}$. 
Hence the Zassenhaus conjecture also holds for central simple
Lie algebras over an arbitrary field of characteristic $p>3$. \\[0.2cm]
However, in characteristic $p=3$ there is one known counterexample to the Zassenhaus conjecture. The same is true for
$p=2$. We will show in the next section that there exists a whole family of counterexamples for $p=3$ of dimension
$3^{n+1}-2$ for all $n\ge 1$.

\section{Simple modular Lie algebras in characteristic three}

We want to study the Zassenhaus conjecture for simple modular Lie algebras of characteristic $p=3$. For the theory
of modular Lie algebras, see for example \cite{SEL}.
First we recall that there is a counterexample, see \cite{BU67}, Proposition $3.5$.

\begin{prop}\label{3.1}
Let $\F$ be a field of characteristic $p=3$. Then the derivation algebra of $\Lg=\mathfrak{psl}_3(\F)$ is isomorphic to
the exceptional Lie algebra $\Lg_2$, and the quotient by $\ad(\Lg)\cong \Lg$ is given by $\Out(\Lg)\cong \Lg$. In particular
the outer derivation algebra of $\Lg$ is simple and non-solvable.
\end{prop}

The next question then is, whether or not there are more counterexamples in characteristic $p=3$. 
Here we distinguish {\em Lie algebras of standard type} (i.e., Lie algebras of classical or Cartan type) and
{\em Lie algebras of non-standard type}.

\subsection{Classical type}
For $p>3$ the list of classical simple modular Lie algebras is given by
\[
\mathfrak{sl}_n(\F), p\nmid n, \; \mathfrak{psl}_n(\F),p\mid n, \; \mathfrak{so}_n(\F), \mathfrak{sp}_{2n}(\F),\Lg_2,\Lf_4,
\Le_6,\Le_7,\Le_8.
\]  
For $p=3$ these Lie algebras are still simple, except for $\Lg_2$ and $\Le_6$. In fact, $\Lg_2$ has a simple ideal
$I\cong \mathfrak{psl}_3(\F)$, generated by the short roots, with $\Lg_2/I\cong I$. This leads to the counterexample mentioned above.
The algebra $\Le_6$ has a $1$-dimensional center so that $\Le_6/\Lz$ is a simple modular Lie algebra of dimension $77$ in
characteristic $3$. Its derivation algebra is abelian, so that we do not obtain another counterexample.
It turns out that for classical simple Lie algebras in characteristic $3$ there are no counterexamples,
except for $\Lg_2/I \cong  \mathfrak{psl}_3(\F)$ discussed above. Indeed, we have the following results, see \cite{BU67}:

\begin{prop}
Let $\F$ be a field of characteristic $3$ and $\Lg$ be a simple Lie algebra of
classical type different from $\mathfrak{psl}_{3m}(\F)$ and $\Le_6/\Lz$. Then $\Out(\Lg)=0$.
\end{prop}

\begin{prop}
Let $\F$ be a field of characteristic $3$. Then we have
$\Der (\mathfrak{psl}_{3m}(\F))\cong \mathfrak{pgl}_{3m}(\F)$ for all $m\ge 2$. Hence $\Out(\mathfrak{psl}_{3m}(\F))\cong \F$
is abelian for all $m\ge 2$. Also $\Out(\Le_6/\Lz)$ is abelian.
\end{prop}

\subsection{Cartan  type}
The list of simple modular Lie algebras of {\em Cartan type} for $p>3$ is given by
the {\em graded} simple Lie algebras of Cartan type
\[
W(m;\underline{n}),\; S(m;\underline{n})^{(1)},\; H(2r;\underline{n})^{(2)},\; K(2r+1;\underline{n})^{(1)},
\]
and their filtered deformations. Here $m\in \N$, $\underline{n}:=(n_1,\ldots ,n_m)\in \N^m$ and
$\abs{\un}:=n_1+\cdots +n_m$. \\[0.2cm]
These algebras are called {\em Witt algebras, special algebras, Hamiltonian algebras and contact algebras}.
They are the finite-dimensional versions defined over a field $\F$ of characteristic $p>0$ of the infinite-dimensional
Lie algebras of characteristic zero occurring in E. Cartan's work of $1909$ on pseudogroups in differential geometry.
For the precise definition of these algebras see H. Strade's book \cite{STR1}. All these algebras are still simple
for characteristic $p=3$, where we need $m\ge 3$ for the special algebras.
The dimensions of these algebras are given by
\begin{align*}
\dim W(m;\underline{n}) & = m\cdot p^{\abs{\un}}, \\
\dim S(m;\underline{n})^{(1)} & = (m-1)(p^{\abs{\un}}-1)\\
\dim H(2r;\underline{n})^{(2)} & = p^{\abs{\un}}-2, \\
  \dim K(2r+1;\underline{n})^{(1)} & =\begin{cases} p^{\abs{\un}},  \hspace*{0.7cm} \text{ if } 2r+1\not\equiv -3 \bmod p,\\
p^{\abs{\un}}-1, \text{ if } 2r+1\equiv -3 \bmod p.\end{cases}
\end{align*}
The derivation algebras have been computed for a field $\F$ of characteristic $p\ge 3$, see
Theorem $7.1.2$ in \cite{STR1}. In particular, the result for $p>3$ still holds for $p=3$, except for the Hamiltonian algebras.
So it follows from the work of Celousov \cite{CEL}, that the Zassenhaus conjecture is true for Witt algebras, special
algebras and contact algebras for an algebraically closed field of characteristic $p\ge 3$. 
However, there are new counterexamples in the Hamiltonian case for $p=3$. 
The following table gives a survey. 
\vspace*{0.5cm}
\begin{center}
\begin{tabular}{c|cccc}
$\Lg$ & conditions & $\dim \Der(\Lg)$ & $\dim \Out(\Lg)$ & conjecture \\[2pt]
\hline
$W(m;\underline{n})$ & $p\ge 2$ &  $m ( p^{\abs{\un}}-1) +\abs{\un}$  & $\abs{\un}-m$  & $\checkmark$  \\[4pt]
$S(m;\underline{n})^{(1)}$ & $p>0,m\ge 3$ & $(m-1)(p^{\abs{\un}}-1)+\abs{\un}+1$   & $\abs{\un}+1$  & $\checkmark$  \\[4pt]
$H(2;(1,1))^{(2)}$ & $p=3$  & $14$ & $7$ & $-$  \\[4pt]
$H(2;(1,n_2))^{(2)}$ & $p=3,n_2>1$  & $3^{n_2+1}+n_2+2$ & $n_2+4$ & $-$  \\[4pt]
$H(2r;\underline{n})^{(2)}$ & $p>3$, or $p=3,r>1$,  & $p^{\abs{\un}}+\abs{\un}$ & $\abs{\un}+2$ & $\checkmark$  \\  
                      & or $p=3,r=1,1<n_1\le n_2$ & & &  \\[4pt]
$K(2r+1;\underline{n})^{(1)}$ & $p>2, p\nmid 2r+4$ & $p^{\abs{\un}}+\abs{\un}-2r-1$ & $\abs{\un}-(2r+1)$ & $\checkmark$  \\[4pt]  
$K(2r+1;\underline{n})^{(1)}$ & $p>2, p\mid 2r+4$ &  $p^{\abs{\un}}+\abs{\un}-2r-1$ & $\abs{\un}-2r$  & $\checkmark$  \\  
\end{tabular}
\end{center}
\vspace*{0.5cm}
Note that we also have 
\[
H(2;(1,n_2))^{(2)}\cong H(2;(n_1,1))^{(2)}
\]
for $p\ge 3$, see \cite{STR1}, $(3)$ on page $199$. \\[0.2cm]
We have first guessed these results for $p=3$ in low dimensions by doing a computation with GAP. 
In fact, we computed the dimensions of the derived series of the outer derivation algebras for the 
Hamiltonian algebras $H(2r;\underline{n})^{(2)}$ in a few cases. The following table shows the results.
The last computation was only possible on the CoCalc server of Anton Mellit, with
$192$ GB RAM. 
\vspace*{0.5cm}
\begin{center}
\begin{tabular}{c|cccc}
$\Lg$ & $\dim(\Lg)$ & $\dim \Der(\Lg)$ & $\dim \Out (\Lg)^{(i)}$ & $\Out(\Lg)$ \\[4pt]
\hline
$H(2;(1,1))^{(2)}$     & $7$   & $14$  & $(7,7,\ldots )$ & simple \\[4pt]
$H(2;(1,2))^{(2)}$     & $25$  & $31$  & $(6,5,5,\ldots )$ &  non-solvable \\[4pt]
$H(2;(1,3))^{(2)}$     & $79$  & $86$  & $(7,5,5,\ldots )$ & non-solvable \\[4pt]
$H(2;(2,2))^{(2)}$     & $79$  & $85$  & $(6,3,1,0)$  & solvable \\[4pt]
$H(4;(1,1,1,1))^{(2)}$ & $79$  & $85$  & $(6,4,0)$ &  solvable  \\[4pt]
$H(2;(2,3))^{(2)}$     & $241$ & $248$ & $(7,3,1,0)$ &  solvable                                           
\end{tabular}
\end{center}
\vspace*{0.5cm}
In order to prove our results, let us introduce further notations. 
Let $\F$ be a field of characteristic $p>2$. Denote by $\CO(m)$ the associative and commutative algebra
with unit element over $\F$ defined by generators $x_i^{(r)}$ for $r\ge 0$ and $1\le i\le m$, and relations
\[
x_i^{(0)}=1,\quad x_i^{(r)}x_i^{(s)}=\binom{r+s}{r}x_i^{(r+s)}
\]  
for $r,s\ge 0$. Put $x_i:=x_i^{(1)}$ and $x^{(a)}:=x_1^{(a_1)}\cdots x_m^{(a_m)}$ for a tuple $a=(a_1,\ldots ,a_m)\in \N^m$.
Then the {\em divided power algebra} of dimension $p^{\abs{\un}}$ is defined by
\[
\CO(m;\underline{n}):={\rm span} \{x^{(a)}\mid 0\le a_i<p^{n_i}\}.
\]  
The product is given by
\[
x^{(a)}x^{(b)}:=\binom{a+b}{b}x^{(a+b)},
\]
where $\binom{a}{b}=\prod_{i=1}^m\binom{a_i}{b_i}$ and $x^{(c)}=0$ if $c_i\ge p^{n_i}$ for some $c_i$.
For each $i$ denote by $\partial_i$ the derivation of the algebra $\CO(m)$ given by
\[
\partial_i(x_j^{(r)})=\de_{i,j}x_j^{(r-1)}. 
\]
The {\em generalized Jacobson-Witt algebra} is defined by
\[
W(m,\underline{n}):=\sum_{i=1}^m \CO(m;\underline{n}) \partial_i,
\]
together with the Lie bracket
\[
[x^{(a)}\partial_i,x^{(b)}\partial_j]=\binom{a+b-\ep_i}{a}x^{(a+b-\ep_i)}\partial_j-\binom{a+b-\ep_j}{b}x^{(a+b-\ep_j)}\partial_i
\]
where $\ep_i=(\de_{i,1},\ldots ,\de_{i,m})\in \N^m$. \\[0.2cm]
Consider the linear operator $D_H\colon \CO(2r;\underline{n})\ra W(2r;\underline{n})$ defined by
\[
D_H(x^{(a)})=\sum_{i=1}^{2r}\sigma(i)\partial_i(x^{(a)})\partial_{i'},
\]  
where
\[
\sigma (i):=\begin{cases} 1,  \hspace*{0.32cm} \text{ if } 1\le i\le r,\\
-1, \text{ if } r+1\le i\le 2r,\end{cases}
\]
and
\[
i':=\begin{cases} i+r,  \text{ if } 1\le i\le r,\\
i-r, \text{ if } r+1\le i\le 2r.\end{cases}
\]
The {\em Hamiltonian algebra} is defined by
\[
H(2r;\underline{n})^{(2)}={\rm span} \{D_H(x^{(a)})\mid 0<a<\tau(\underline{n})\},
\]
where $\tau(\underline{n})=(p^{n_1}-1,\ldots ,p^{n_m}-1)\in \N^m$. The Lie bracket is given by
\[
[D_H(x^{(a)}),D_H(x^{(b)})]=D_H(D_H(x^{(a)})(x^{(b)})).
\]  

The main result of this paper is that we obtain an infinite family of counterexamples to the Zassenhaus conjecture, which contains
the known counterexample $\mathfrak{psl}_3(\F)$ as the smallest case $n=1$:

\begin{thm}\label{3.4}
For all $n\ge 1$ the simple modular Lie algebra $H(2;(1,n))^{(2)}$ of dimension $3^{n+1}-2$ in characteristic $3$
does not have a solvable outer derivation algebra.
\end{thm}  

\begin{proof}
We will use the basis of $\Lg=H(2;(1,n))^{(2)}$ given above, for the special case  of $p=3$, $m=2$,
and $\underline{n}=(n_1,n_2)=(1,n)$. For $x^{(\alpha)}$ we will write $x_1^{a}x_2^{b}$. Then the explicit
Lie brackets are given by
\[
[D_H(x_1^ax_2^b), D_H(x_1^cx_2^d)]=f_{a,b,c,d}\cdot D_H(x_1^{a+c-1}x_2^{b+d-1}),
\]  
where
\[
f_{a,b,c,d}:= e_ae_d\cdot \binom{a+c-1}{a-1} \binom{b+d-1}{d-1}
-e_be_c\cdot \binom{a+c-1}{c-1} \binom{b+d-1}{b-1},
\]
with $e_k:= 1-\delta_{k,0}$. \\[0.2cm]
Let us order the basis elements $D_H(x_1^ax_2^b)$ of $\Lg$ with respect to the formal 
exponents as follows:
\[
D_H(x_1)\prec D_H(x_1^2)\prec D_H(x_2)\prec D_H(x_1x_2)\prec D_H(x_1^2x_2) \prec \cdots ,
\]  
so that we can write a general inner derivation $D \in \Der(\Lg)$ as
\[
\al \cdot \ad (D_H(x_1))+\be \cdot\ad (D_H(x_1^2)) +\ga \cdot \ad (D_H(x_2)) + \de \cdot \ad (D_H(x_1x_2))
+\ep \cdot \ad (D_H(x_1^2x_2))+\cdots 
\]  
Using the Lie brackets, the matrix of $D$ with respect to this ordered basis is of the form
\[
D=  \left(\begin{array}{@{}ccc|ccc|ccc@{}}
    -\de   & -\ga   &   &    &      &  & & & \\
    -\ep   & \de    & 0 &    &      &  & & & \\ 
           & 0      & \de    & -\ga &      & & & & \\ \hline
           &        & \ep    & 0    & -\ga & & & &\\
           &        &        & \ep  & -\de &  0 & & &\\ 
           &        &        &      & 0    & -\de & & & \\ \hline
\vdots     & \vdots & \vdots &      &      &      & \vdots & \vdots & \vdots \\  \hline
0          &        &        &      &      &      &  & &  \\ 
0          &        &        &      &      &      &  & & \\
0          & 0      & 0      &      &      &      &  & & \\
\end{array}\right).
\]
For $n=1$ we have $\Lg=H^2(2;(1,1))^{(2)}\cong \mathfrak{psl}_3(\F)$, where we already know that $\Out(\Lg)\cong \Lg$
is not solvable, see Proposition $\ref{3.1}$. So we may assume that $n>1$.
Consider the linear maps $E,F,H\in \End(\Lg)$ defined by 
\begin{align*}
	E&\colon \Lg \to \Lg, \quad D_H(x_1^ax_2^b)\mapsto \delta_{a,2}\cdot D_H(x_2^{b+1}),\\
	F&\colon \Lg \to \Lg, \quad D_H(x_1^ax_2^b)\mapsto \delta_{a,0}\cdot  D_H(x_1^2x_2^{b-1}),\\
	H&\colon \Lg \to \Lg, \quad D_H(x_1^ax_2^b)\mapsto (1-a)\cdot D_H(x_1^ax_2^b).
\end{align*}
We claim that $E,F,H\in \Der(\Lg)$ are derivations of $\Lg$. This follows easily from a direct computation.
Indeed, we have
\begin{align*}
E([D_H(x_1^ax_2^b),D_H(x_1^cx_2^d)]) & =f_{a,b,c,d}\cdot E(D_H(x_1^{a+c-1}x_2^{b+d-1}))\\
    & =f_{a,b,c,d}\cdot \delta_{a+c-1,2}\cdot D_H(x_2^{b+d})\\
    & = {{b+d}\choose{b}}\cdot \delta_{(a,c),(1,2)}\cdot D_H(x_2^{b+d})-{{b+d}\choose{b}}\cdot \delta_{(a,c),(2,1)}\cdot D_H(x_2^{b+d})\\
    & = -\delta_{a,2} e_c\cdot {{b+d}\choose{b}} D_H(x_1^{c-1}x_2^{b+d})+ \delta_{c,2}e_a\cdot {{b+d}\choose{b}} D_H(x_1^{a-1}x_2^{b+d})\\
    & = \delta_{a,2}\cdot f_{0,b+1,c,d}\cdot D_H(x_1^{c-1}x_2^{b+d}) +  \delta_{c,2}\cdot f_{a,b,0,d+1}\cdot D_H(x_1^{a-1}x_2^{b+d}) \\[0.1cm]
    & = [\delta_{a,2}\cdot D_H(x_2^{b+1}),D_H(x_1^{c}x_2^{d})]+[D_H(x_1^{a}x_2^{b}),\delta_{c,2}\cdot D_H(x_2^{d+1})]\\[0.1cm]
    & = [E(D_H(x_1^{a}x_2^{b})),D_H(x_1^{c}x_2^{d})] + [D_H(x_1^{a}x_2^{b}),E(D_H(x_1^{c}x_2^{d}))].\\ 
\end{align*}
Here we have used that $2=-1$ in $\F$ and Pascal's identity
\[
\binom{b+d-1}{d-1}+\binom{b+d-1}{d}=\binom{b+d}{d}.	
\]
A similar computation shows that also $F$ and $H$ are derivations. On the other hand, this follows anyway, because $F$ coincides
with the restriction of the inner derivation $\ad (D_H(x_1^3))$ of the larger Lie algebra $H(2;(1,n))$, and $H$ coincides with the
commutator $[E,F]$, and hence is a derivation. It is easy to see that we have
\[
[E,H]=E=-2E,\; [F,H]=-F=2F, [E,F]=H.
\]  
Thus $(E,F,H)$ forms an $\mathfrak{sl}_2(\F)$-triple in $\Der(\Lg)$, i.e., the subalgebra $\Ls$ of $\Der(\Lg)$ generated by
$E,F,H$ is isomorphic to $\mathfrak{sl}_2(\F)$. Now the matrix of $\la E+\mu F+\nu H$ with respect to the ordered basis of $\Lg$
has the form
\[
D=  \left(\begin{array}{@{}ccc|ccc|ccc@{}}
    0      &        &        &    &         &  & & & \\
           & -\nu   & \mu    &    &         &  & & & \\ 
           & \la    & \nu    &    &         & & & & \\ \hline
           &        &        & 0    &       & & & &\\
           &        &        &      & -\nu  &  \mu & & &\\ 
           &        &        &      & \la   &  \nu & & & \\ \hline
\vdots     & \vdots & \vdots & \vdots & \vdots & \vdots  & \vdots & \vdots & \vdots \\  
           &        &        &      &      &      &  & &  \\ 
\end{array}\right).
\]
Comparing this with the form for the general inner derivation $D$ we conclude that the subalgebra $\Ls$ satisfies
$\Ls\cap \ad (\Lg)=0$. Hence $\Out(\Lg)$ contains the subalgebra
\[
(\Ls+\ad(\Lg))/\ad(\Lg)\cong \Ls/\Ls\cap \ad(\Lg)\cong \Ls\cong \mathfrak{sl}_2(\F).
\]  
Thus  $\Out(\Lg)$ is not solvable.
\end{proof}  

So we have obtained an infinite family of counterexamples. In addition, we can be more precise about the structure of the
outer derivation algebra of $H(2;(1,n))^{(2)}$. Denote by $V(2)$
the natural representation of $\mathfrak{sl}_2(\F)$. Then the Lie algebra $\mathfrak{sl}_2(\F)\ltimes V(2)$ in characteristic
$3$ has a basis $(e_1,\ldots ,e_5)$ with Lie brackets
\begin{align*}
[e_1,e_2] & = e_3,    & [e_2,e_3]  & = 2e_2,  & [e_3,e_4] & = e_4,\\
[e_1,e_3] & = e_1,    & [e_2,e_4]  & = e_5,   & [e_3,e_5] & = 2e_5. \\
[e_1,e_5] & = e_4,    &            
\end{align*}

\begin{thm}\label{3.5}
Let $n>1$. Then the outer derivation algebra of $H(2;(1,n))^{(2)}$ in characteristic $3$ is isomorphic to
$(\mathfrak{sl}_2(\F)\ltimes V(2))\oplus \F^{n-1}$. 
\end{thm}
  
\begin{proof}
Let $\Lg=H(2;(1,n))^{(2)}$. According to \cite{STR1}, Theorem $7.1.2$, $(3)$ part $(b)$ on page 
$358$ we have
\[
\Der(\Lg) \cong CH(2;(1,n))+ \sum_{i=1}^{n-1} \F \cdot \partial_2^{3^i}+ \F\cdot d,
\]
where $d$ is the derivation which we called $F$ in the proof of Theorem $\ref{3.4}$, and
\[
CH(2;(1,n))=H(2;(1,n))\oplus \F\cdot (x_1\partial_1+x_2\partial_2).
\]  
We have $\dim CH(2;(1,n))=3^{n+1}+2$, see \cite[page 273]{KS}, so that we obtain $\dim \Der(\Lg)=3^{n+1}+n+2$
and $\dim \Out(\Lg)=n+4$. Consider the linear maps given by
\begin{align*}
V\ &\colon\ L\to L, \quad D_{H}(x_1^ax_2^b) \mapsto \delta_{b,0} \cdot D_H(x_1^{a-1}x_2^{3^n-1}), \\ 
W \ &\colon \ L\to L, \quad D_{H}(x_1^ax_2^b) \mapsto \delta_{a+b,1}\cdot(-1)^a\cdot D_H(x_1^{a+1}x_2^{b+3^n-2}).
\end{align*}
They are derivations of $\Lg$, because each of them is a restriction of inner derivations of the 
larger Lie algebra
$H(2;(1,n))$ to $\Lg$, namely of $\ad(D_H(x_2^{3^n}))$, respectively of $\ad(D_H(x_1^2x_2^{3^n-1}))$. By a computation we see that
\[
[E,W]=V,\; [F,V]=W,\; [H,V]=V,\; [H,W]=2W,
\]
where $E,F,H$ are the derivations of $\Lg$ given in the proof of Theorem $\ref{3.4}$. Hence the subalgebra $\Lt$ of $\Der(\Lg)$
generated by $E,F,H,V,W$ is isomorphic to $\mathfrak{sl}_2(\F)\ltimes V(2)$. \\
The matrix of $\la E+\mu F+\nu H + \eta V+\xi W$ with respect to the ordered basis of $\Lg$ is of the form
\[
D=  \left(\begin{array}{@{}ccc|ccc|ccc@{}}
    0      &        &        &    &         &  & & & \\
           & -\nu   & \mu    &    &         &  & & & \\ 
           & \la    & \nu    &    &         & & & & \\ \hline
           &        &        & 0    &       & & & &\\
           &        &        &      & -\nu  &  \mu & & &\\ 
           &        &        &      & \la   &  \nu & & & \\ \hline
\vdots     & \vdots & \vdots & \vdots & \vdots & \vdots  & \vdots & \vdots & \vdots \\  \hline
-\xi       & 0      & 0      &      &      &      & \ddots  & &  \\ 
\eta       & 0      & 0      &      &      &      &         & \ddots & \\
0          & \eta   & \xi    &      &      &      &      &   & \ddots \\
\end{array}\right)
\]
Comparing with the matrix $D$ of inner derivations (see the proof of Theorem $3.4$) we obtain $\Lt \cap \ad(\Lg)=0$, so that
$\Out(\Lg)$ has a subalgebra isomorphic to $\mathfrak{sl}_n(\F)\ltimes V(2)$. We claim that the derivations
$\partial_2^{3^i}$ belong to the center of $\Out(\Lg)$. Indeed, they commute pairwise, and they commute with $E,F,H$.
Furthermore we have, using also \cite[Lemma 2.1.2(1), page 61]{STR1},
\begin{align*}
[\partial_2^{3^i},V]&=\ad(D(x_2^{3^n-3^i})),\\
[\partial_2^{3^i},W]&=\ad(D(x_1^2x_2^{3^n-3^i-1})),
\end{align*}
for $i=1,\dots,n-1$. This implies that $\Out(\Lg)\cong \Lt \oplus \F^{n-1}$, where 
$\Lt\cong \mathfrak{sl}_2(\F)\ltimes V(2)$. 
\end{proof}

We will show now that the remaining cases for the Hamiltonian Lie algebras
$H(2r;\underline{n})^{(2)}$ do not provide new counterexamples to the Zassenhaus conjecture
for $p=3$. We have two cases, namely first $r>1$, and secondly $r=1$ and $1<n_1\le n_2$, where
$\underline{n}=(n_1,n_2)\in \N^2$. Let $\Lh_3(\F)$ be the Heisenberg Lie algebra over $\F$ with basis $\{e_1,e_2,e_3\}$ and
Lie bracket $[e_1,e_2]=e_3$. Recall that a Lie algebra over a field $\F$ is called {\em almost abelian} 
if it is nonabelian and has an ideal of codimension $1$. Hence every almost abelian Lie algebra can be 
written as $\F^r\rtimes \F$, and is $2$-step solvable.

\begin{thm}\label{3.6}
Let $\Lg$ be the Hamiltonian Lie algebra $H(2r;\underline{n})^{(2)}$ over a field $\F$ of characteristic
$p=3$. Then, for $r>1$ the outer derivation algebra $\Out(\Lg)$ is $2$-step solvable, and for $r=1$, $1<n_1\le n_2$,
it is $3$-step solvable. More precisely, we have  
\[
\Out(\Lg)\cong 
\begin{cases} (\Lh_3(\F)\rtimes \F)\oplus \F^{\abs{\un}-2}, \hspace*{0.13cm} \text{ if } r=1,\; 1<n_1\le n_2,\\
(\F^{2r+1}\rtimes \F)\oplus \F^{\abs{\un}-2r},  \text{ if } r>1, r\equiv 0 \bmod 3, \\
(\F^{2r+1}\rtimes \F)\oplus \F^{\abs{\un}-2r},  \text{ if } r>1, r\equiv 1 \bmod 3, \\  
(\F^{2r}\rtimes \F)\oplus \F^{\abs{\un}-2r+1},   \text{ if } r>1, r\equiv 2 \bmod 3. \\
\end{cases}
\]
Here in the first case $\F$ acts on $\Lh_3(\F)$ by the derivation $D={\rm diag}(1,1,-1)$, in the second case
$\F$ acts on $\F^{2r+1}$ by the derivation $D=\id$, in the third case $\F$ acts on $\F^{2r+1}$ by the derivation
$D={\rm diag}(1,\ldots ,1,-1)$, and in the last case  $\F$ acts on $\F^{2r}$ by the derivation $D=\id$.
\end{thm}

\begin{proof}
Let us write $x^a$ for $x^{(a)}= x_1^{a_1}\cdots x_m^{a_m}$ and 
\[
\tau=(3^{n_1}-1,\dots,3^{n_m}-1)\in \N^m.
\]
By ~\cite[Theorem 7.1.2(3)(b), page 358]{STR1}, the structure of $\Der H(2r;\underline{n})^{(2)}$ is given by
\[
  \Der H(2r;\underline{n})^{(2)}\cong CH(2r;\underline{n})^{(2)} \oplus \sum_{i=1}^{2r}\sum_{0<j_i<n_i}\F \cdot \partial_i^{j_i},
\]
where
\[
  CH(2r;\underline{n})^{(2)} = H(2r;\underline{n})\oplus \F\cdot \left(\sum_{i=1}^{2r}x_i\partial_i\right).
\]
So we obtain the following dimensions:
\begin{align*}
  \dim \Der H(2r;\underline{n})^{(2)} & = \dim H(2r;\underline{n}) + 1 + |\un|-2r \\
                                     & = (3^{|\un|}-2+2r+1) + 1 + |\un| - 2r \\
                                     & = 3^{|\un|} + |\un|,
\end{align*}
see also~\cite[page 273]{KS}. So we have 
\begin{align*}
\dim \Out H(2r;\underline{n})^{(2)} & = \dim \Der H(2r;\underline{n})^{(2)}- \dim H(2r;\underline{n})^{(2)}\\
                                    & =3^{|\un|} + |\un|-(3^{|\un|}-2)\\
                                    & = |\un|+2.
\end{align*}
Consider the restrictions to $H(2r;\underline{n})^{(2)}$ of the derivations $\ad(D_H(x_i^{p^{n_i}}))$ and
$\ad(D_H(x^{\tau}))$ of the
larger Lie algebra $H(2r;\underline{n})$. They are given explicitly as the linear maps
\begin{align*}
A_i & \colon H(2r;\underline{n})^{(2)} \to H(2r;\underline{n})^{(2)},\quad D_H(x^a) \mapsto
  \delta_{a_i,0}\cdot \sigma(i)\cdot D_H(x^{a+(\tau_i-a_i)\ep_i-\ep_{i'}})\\
  B & \colon H(2r;\underline{n})^{(2)} \to H(2r;\underline{n})^{(2)} ,\quad  D_H(x^a) \mapsto \delta_{|a|,1}\cdot
   \sigma(k) \cdot D_H(x^{\tau-\ep_k})
\end{align*}
for $i=1,\dots, 2r$, and where $k\in\{1,\ldots ,2r \}$ is the only index such that $a_{k'}\neq 0$. Recall the definition
of $k'$ before Theorem $\ref{3.4}$. It is clear that $A_i,B\in \Der H(2r;\underline{n})^{(2)}$.
Moreover the derivations $C:= \sum_{i=1}^{2r}x_i\partial_i$ and $D_{i,j_i}:= \partial_i^{j_i}$ for $i=1,\cdots ,2r$
and  $0<j_i<n_i$ for each $i$ are explicitly given by
\begin{align*}
C & \colon H(2r;\underline{n})^{(2)} \to H(2r;\underline{n})^{(2)}, \quad D_H(x^a) \mapsto (|a|-2)\cdot D_H(x^a)\\
D_{i,j_i} & \colon H(2r;\underline{n})^{(2)} \to H(2r;\underline{n})^{(2)}, \quad D_H(x^a) \mapsto D_H(x^{a-p^{j_i}\ep_i}).
\end{align*}
We claim that
\[
\{A_1,\ldots , A_{2r},B,C,D_{1,1},\ldots D_{1,n_1-1},\ldots ,D_{2r,1},\ldots ,D_{2r,n_{2r}-1}\}
\]
are representatives of a basis of $\Out H(2r;\underline{n})^{(2)}$. Its cardinality is given by
$2r+2+\sum_{i=1}^{2r}n_i -2r=\abs{\un}+2$. The arguments are the same
as used in the proofs of Theorem \ref{3.4} and Theorem \ref{3.5}, i.e., one can easily check that the
intersection of the linear span of these derivations and $\ad H(2r;\underline{n})^{(2)}$ is zero. Indeed, this follows
just from comparing the images of $D_H(x_i)$ for  $i=1,\dots, 2r$, under a general inner derivation and
$\sum_{i=1}^{2r}\alpha_i A_i + \beta B + \gamma C + \sum_{i=1}^{2r}\sum_{0<j_i<n_i} \delta_{i,j_i} D_{i,j_i}$. The projections
onto $\Out H(2r;\underline{n})^{(2)}$ of $A_i$, $B$, $C$ and $D_{i,j_i}$ are then  $|n|+2$ linearly independent
derivations which therefore constitute a basis of $\Out H(2r;\underline{n})^{(2)}$. \\[0.2cm]
It is straightforward to compute the Lie brackets between the representatives in $\Der H(2r;\underline{n})^{(2)}$
of the basis vectors of $\Out H(2r;\underline{n})^{(2)}$.
The nonzero brackets are given as follows, with $1\le i<i'\le 2r$,

\begin{align*}
[A_i,A_{i'}] & =\begin{cases} B & \text{ if } r=1 \\ \ad D_H(x_i^{\tau_i}x_{i'}^{\tau_{i'}}) & \text{ if } r>1 \end{cases} \\
[A_i,C]         & = -A_i, \\
[A_i,D_{i,j_i}] & =-\ad D_H(x_i^{\tau_i-p^{j_i}+1}),\\
[B,C]          & = (2r-1)B,\\  
[B,D_{i,j_i}]   & =-\ad D_H(x^{\tau-p^{j_i}\ep_i}).\\
\end{align*}

Note that $[B,C]=0$ for the case $r\equiv 2\bmod 3$. For $r>1$, the Lie brackets yield a direct sum of an
almost abelian Lie algebra $\F^{2r+1}\rtimes \F$ (or $\F^{2r}\rtimes \F$ for $r\equiv 2\bmod 3$), and an abelian
Lie algebra. Hence $\Out(\Lg)$ is $2$-step solvable in this case. For $r=1$ we have $[C,B]=-B$, $[C,A_i]=A_i$ for
$i=1,2$, and $[A_1,A_2]=B$, so that
\[
\Out H(2r;\underline{n})^{(2)}\cong {\rm span}(A_1,A_2,B,C) \oplus {\rm span}(D_{i,j_i})
\cong  (\Lh_3(\F)\rtimes \F)\oplus \F^{\abs{\un}-2}.
\]
The ideal $\La={\rm span}(A_1,A_2,B,C)$ satisfies $\La^{(1)}={\rm span}(A_1,A_2,B)$, $\La^{(2)}={\rm span}(B)$ and $\La^{(3)}=0$.
Thus  $\Out(\Lg)$ is $3$-step solvable for $r=1$. 
\end{proof}

\subsection{Non-standard type}

There are several simple modular Lie algebras over a field of characteristic $3$ that are neither of
classical nor Cartan type. For example, the $1$-parameter family of $10$-dimensional Kostrikin algebras $L(\ep)$, the Ermolaev algebras
$R(\underline{n})$, the Brown-Kuznetsov algebras $T(n)$, and the Skryabin algebras $X(\underline{n})$ and $Y(\underline{n})$.
Chan Nam Zung studied their properties in \cite{CNZ}, published in $1993$. He computed the outer 
derivation algebras of these algebras. It turns out that we do not obtain any new counterexample to the Zassenhaus conjecture.
The following table gives a survey.

\vspace*{0.5cm}
\begin{center}
\begin{tabular}{c|cccc}
$\Lg$ & conditions & $\dim (\Lg)$ & $\dim \Out(\Lg)$ & $\Out(\Lg)$ \\[2pt]
\hline
$L(\ep)$ & $\ep\in \F$  &  $10$  & $0$  & abelian  \\[4pt]
$R(\underline{n})$      & $\underline{n}=(n_1,n_2)\in \N^2$ & $3^{\abs{\un}+1}-1$   & $\abs{\un}+1$  & abelian  \\[4pt]
$T(n)$ & $n\in \N$  & $2\cdot 3^{n+1}$ & $n-1$ & abelian  \\[4pt]
$X(\underline{n})$   & $\underline{n}=(n_1,n_2,n_3)\in \N^3$  & $3^{\abs{\un}+1}-4$   & $\abs{\un}+1$ & solvable \\[4pt]  
$Y(\underline{n})$   & $\underline{n}=(n_1,n_2,n_3)\in \N^3$ & $2\cdot 3^{\abs{\un}+1}$ & $\abs{\un}-3$ & abelian  \\  
\end{tabular}
\end{center}
\vspace*{0.5cm}

However, there are three further infinite families of simple Skryabin algebras in characteristic three, 
denoted by $Z'(\underline{n})$, and $X_i(\underline{n},\om)$, for $i=1,2$ of type $1$ and type $2$, see \cite{SKR}.
Zung does not determine the outer derivation algebras of these families in \cite{CNZ}. He mentions that the determination
for $Z'(\underline{n})$ is still an open problem. However, this  was solved $2001$ in \cite{KUM}. The outer derivation
algebra is abelian. Unfortunately we could not find a result for the algebras $X_i(\underline{n},\om)$. 
But we believe that the outer derivation algebra will be solvable, too. Let us explain the result of 
\cite{KUM} on the derivation algebra of $Z'(\underline{n})$.
In the construction of the Lie algebra $Z'(\underline{n})$, Skryabin introduces a Lie algebra
$Z(\underline{n})$ of dimension $3^{\abs{\un}+2}+1$ with
\[
Z'(\underline{n})=[Z(\underline{n}),Z(\underline{n})]. 
\]
Using this notation, the result of \cite{KUM} is as follows, see Corollary $1$ on page $3925$.

\begin{prop}
Let $\Lg=Z'(\underline{n})$, with $\underline{n}=(n_1,n_2,n_3)\in \N$. Then we have
\[
\Der(\Lg)\cong \ov{\Lg_{\ov{0}}}+Z(\underline{n}).
\]
\end{prop}

Here $\Lg_{\ov{0}}\cong W(3,\underline{n})$ and $\ov{\Lg_{\ov{0}}}$ denotes the $p$-closure of $\ad(\Lg_{\ov{0}})$ in
$\Der(\Lg)$. This implies that $\Out(\Lg)$ is abelian, since
\[
[\Der(\Lg),\Der(\Lg)]\subseteq [Z(\underline{n}),Z(\underline{n})]=Z'(\underline{n})\cong \ad (\Lg).
\]  
Furthermore we have the $10$-dimensional simple Lie algebras $L(\ep,\de,\rho)$ in characteristic three 
of Kostrikin \cite{KOS}, which are deformations of the algebras $L(\ep)$.
Here it is known that all derivations are inner. All known simple Lie algebras of dimension $10$ 
for $p=3$ can be realized within the family $L(\ep,\de,\rho)$, see \cite{KOS}, 
but a classification up to isomorphism is still not known. \\[0.2cm]
Finally we have the $8$-dimensional and the $29$-dimensional simple Lie algebras 
$Br_8$ and $Br_{29}$ of Brown \cite{BR3,BR2}. Both Lie algebras are central simple. A direct computation 
shows that the outer derivation algebra is abelian in each case. Surprisingly, $Br_8$ is not mentioned 
in later works on simple Lie algebras of characteristic three. Thus, for the convenience of the reader, let us give all
Lie brackets of $Br_8$ explicitly, with respect to the basis
\[
(x_1,\ldots ,x_8)=(K_{12},K_{21},K_{13},K_{31},K_{23},K_{32},H,K)
\]  
introduced in \cite{BR3} on page $440$:

\begin{align*}
[x_1,x_{2}]   & = x_{7},   & [x_2,x_{6}]  & =2x_{4},   & [x_4,x_{5}]  & = 2x_{2},\\
[x_1,x_{4}]   & = 2x_{6},  & [x_2,x_{7}]  &  = 2x_{2}, & [x_4,x_{7}]  & = x_{4}, \\
[x_1,x_{5}]   & = x_{3},   & [x_2,x_{8}]  &  = 2x_{6}, & [x_5,x_{6}]  & = x_{7}, \\
[x_1,x_7]     & = x_1,     & [x_3,x_{4}]  & = 2x_{7},  & [x_5,x_{7}]  & =x_{5},\\
[x_2,x_{3}]   & = x_{5},   & [x_3,x_{6}]  & = x_{1},   & [x_{5},x_{8}]& = x_{1},\\
[x_2,x_{5}]   & = x_{8},   & [x_3,x_{7}]  & = 2x_{3},  & [x_{6},x_{7}]& = 2x_{6}. \\
\end{align*}

This algebra is central simple and non-restricted. Its outer derivation algebra is $2$-dimensional
and abelian. Note that $Br_8$ is isomorphic to a deformed Hamiltonian algebra 
$H(2; (1,1), \omega)$, where $\omega =(1+x_1^{(2)}x_2^{(2)})(dx_1\wedge dx_2)$. For the family
of simple deformed Hamiltonian algebras $H(2r;\underline{n},\om)$ of dimension 
$p^{\abs{\underline{n}}}-1$ see \cite{STR1}, pp. $340-341$. The following table gives a survey of the 
preceding discussion.
\vspace*{0.5cm}
\begin{center}
\begin{tabular}{c|cccc}
$\Lg$ & conditions & $\dim (\Lg)$ & $\dim \Out(\Lg)$ & $\Out(\Lg)$ \\[2pt]
\hline
$Br_8$ & $-$  &  $8$  & $2$  & abelian  \\[4pt]
$L(\ep,\de,\rho)$ & $\ep,\de,\rho\in \F$ & $10$ & $0$  & abelian  \\[4pt]
$Br_{29}$ & $-$  &  $29$  & $0$  & abelian  \\[4pt]
$Z'(\underline{n})$      & $\underline{n}=(n_1,n_2,n_3)\in \N^3$ & $3^{\abs{\un}+2}-2$   & $\abs{\un}$  & abelian  \\[4pt]
$X_1(\underline{n},\om)$   & $\underline{n}=(n_1,n_2,n_3)\in \N^3$  & $3^{\abs{\un}+1}-3$   & ? & ? \\[4pt]  
$X_2(\underline{n},\om)$   & $\underline{n}=(n_1,n_2,n_3)\in \N^3$ & $3^{\abs{\un}+1}-1$ & ? & ?  \\  
\end{tabular}
\end{center}
\vspace*{0.5cm}
There are other simple Lie algebras for $p=3$, which we have not studied here, e.g.,  
deformed Hamiltonian and special Lie algebras of Cartan type for $p=3$, or other families, where no explicit
realization is known.  

\begin{rem}
We also studied the Zassenhaus conjecture for simple Lie algebras over a field of characteristic $p=2$. Here it was already
known since $1955$ that a simple constituent $J$ of dimension 
$26$ of the Lie algebra $\Lf_4$  provides a counterexample, see \cite{SCT}, and \cite{BU67} for references. 
Note that $J$ is given as the simple ideal in $\Lf_4$ generated by the short roots. We tried to find an infinite 
family of simple Lie algebras such that the algebra $J$ is the lowest-dimensional member.
One possibility is the family of simple Lie algebras $\mathfrak{si}(\mathfrak{sle}(n))$ of 
dimension $2^{2n-1}-2^{n-1}-2$ for $n\ge 3$, see \cite{KOL}, Lemma $2.2.2$. This algebra is denoted by
$\mathfrak{sh}(2n;\underline{m})$ in Purslow's thesis \cite{PUR}, Theorem $5.4.3$. 
We used Purslow's construction for $n=4$, see \cite{PUR} pp. $138-141$, to compute the outer derivation algebra of this
$118$-dimensional algebra. It is a solvable Lie algebra of derived length $5$. So it is not a counterexample, but the
derived length is higher than in all other known cases.
For $n=5$ the algebra has dimension $494$, but we could not compute the derivation algebra so far. \\[0.2cm]
We also tested the table of B. Eick in \cite{EIK} with known simple Lie algebras
up to dimension $20$, but found no counterexample there. There are various families of simple 
Lie algebras of non-standard type, and it seems to be very complicated to obtain an overview on the 
Zassenhaus conjecture here. So far, all families we have been able to
study did not yield a new counterexample.
\end{rem}

\section*{Acknowledgments}
Dietrich Burde is supported by the Austrian Science Foun\-da\-tion FWF, grant I 3248 and grant P 33811. 
Pilar P\'aez-Guill\'an is supported by the Austrian Science Foun\-da\-tion FWF, grant P 33811.
We thank the referee for suggesting several improvements of the text. We also thank Bettina Eick and Tobias Moede for help
with some computations, and Anton Mellit for providing us access to the CoCalc server for GAP computations.

\end{document}